\documentclass[12pt]{amsart}
\usepackage{amsmath,amsfonts,euscript,amscd,amsthm,amssymb,upref,graphics}
\usepackage[all]{xy}

\swapnumbers

\theoremstyle{plain}

\newtheorem{theorem}{Theorem}[section]

\newtheorem{lemma}[theorem]{Lemma}
\newtheorem{corollary}[theorem]{Corollary}

\theoremstyle{definition}

\newtheorem{definition}[theorem]{Definition}

\newtheorem{example}[theorem]{Example}

\newtheorem*{questionn}{Question}

\newtheorem*{conventions}{Conventions}
\newtheorem{subnothing*}[sub]{}

\newtheorem{somefacts}[theorem]{Some facts}

\theoremstyle{remark}

\newtheorem*{remark*}{Remark}
\newtheorem{remark}[theorem]{Remark}

{\begin{enumerate}\setlength{\itemsep}{#1}}{\end{enumerate}}

\newcommand{\Aut}{	\operatorname{{\rm Aut}}}

\newcommand{\Spec}{	\operatorname{{\rm Spec}}}

\newcommand{\trdeg}{	\operatorname{{\rm trdeg}}}
\newcommand{\Frac}{	\operatorname{{\rm Frac}}}

\newcommand{\Pic}{	\operatorname{{\rm Pic}}}
\newcommand{\Cl}{	\operatorname{{\rm Cl}}}

\newcommand{\lnd}{	\operatorname{\mbox{\sc lnd}}}
\newcommand{\klnd}{	\operatorname{\mbox{\sc klnd}}}
\newcommand{\Klnd}{	\operatorname{\mbox{\underline{\sc klnd}\,}}}
\newcommand{\Gal}{	\operatorname{{\rm Gal}}}
\newcommand{\Dan}{	\operatorname{{\Dgoth}}}
\newcommand{\ML}{	\operatorname{{\rm ML}}}

\newcommand{\danml}{	\operatorname{\mbox{\sc danml}}}
\newcommand{\sml}{	\operatorname{\mbox{\sc sml}}}

\newcommand{\setspec}[2]{\big\{\,#1\, \mid \,#2\, \big\}}

\newlength{\mylength}
\newcommand{\Nat}{\ensuremath{\mathbb{N}}}
\newcommand{\Rat}{\ensuremath{\mathbb{Q}}}
\newcommand{\Comp}{\ensuremath{\mathbb{C}}}
\newcommand{\Reals}{\ensuremath{\mathbb{R}}}
\newcommand{\aff}{\ensuremath{\mathbb{A}}}

\newcommand{\bk}{{\ensuremath{\rm \bf k}}}
\newcommand{\ck}{{\bar{\bk}}}
\newcommand{\kk}[1]{\bk^{[#1]}}

\newcommand{\Dgoth}{{\ensuremath{\mathfrak{D}}}}

\newcommand{\pgoth}{{\ensuremath{\mathfrak{p}}}}

\newcommand{\Aeul}{\EuScript{A}}
\newcommand{\Beul}{\EuScript{B}}

\newcommand{\Deul}{\EuScript{D}}

\newcommand{\Neul}{\EuScript{N}}
\newcommand{\Oeul}{\EuScript{O}}

\newcommand{\Reul}{\EuScript{R}}

\newcommand{\isom}{\cong}
\renewcommand{\epsilon}{\varepsilon}
\renewcommand{\phi}{\varphi}
\renewcommand{\emptyset}{\varnothing}

\newcommand{\rien}[1]{}

\begin{document}
\renewcommand{\baselinestretch}{1.07}


\title[Affine surfaces with trivial Makar-Limanov invariant]
{Affine surfaces \\  with trivial Makar-Limanov invariant}

\author{Daniel Daigle}

\address{Department of Mathematics and Statistics\\
	University of Ottawa\\
	Ottawa, Canada\ \ K1N 6N5}

\email{ddaigle@uottawa.ca}

\thanks{Research supported by a grant from NSERC Canada.}


{\renewcommand{\thefootnote}{}
}

\begin{abstract} 
We study the class of $2$-dimensional affine $\bk$-domains $R$
satisfying $\ML(R)=\bk$, where $\bk$ is an arbitrary field of
characteristic zero.
In particular, we obtain the following result:
{\it Let $R$ be a localization of a polynomial ring in
finitely many variables
over a field of characteristic zero. If $\ML(R) = K$ for some field
$K \subset R$ such that $\trdeg_KR=2$,
then $R$ is $K$-isomorphic to $K[X,Y,Z]/(XY-P(Z))$ for some nonconstant
$P(Z) \in K[Z]$.}
\end{abstract}
\maketitle
  
\vfuzz=2pt


\section{Introduction}

Let us recall the definition of the Makar-Limanov invariant:

\begin{definition}\label{fefefefeer}
If $R$ is a ring of characteristic zero,
a derivation $D: R\to R$ is said to be \textit{locally nilpotent}
if for each $r\in R$ there exists $n \in \Nat$ (depending on $r$)
such that $D^n(r)=0$.
We use the following notations:
\begin{align*}
\lnd(R) &= \textrm{set of locally nilpotent derivations $D:R\to R$} \\
\klnd(R) &= \setspec{ \ker D }{ D \in \lnd(R) \text{ and } D \neq 0} \\
\ML(R) &= \bigcap_{ D \in \text{\sc lnd}(R) } \ker(D).
\end{align*}
\end{definition}

We are interested in the class of $2$-dimensional affine $\bk$-domains $R$
satisfying $\ML(R)=\bk$, where $\bk$ is a field of characteristic zero.
The corresponding class of
affine algebraic surfaces was studied by several authors
(\cite{BandML:AffSurfAK},
\cite{Bertin:Pinceaux},
\cite{DaiRuss:HomolPlanes1},
\cite{Dub:CompletionsNormAffSurf},
\cite{Dub:DanFies},
\cite{Gur-Miy:ML},
\cite{Masuda-Miy:QHomPlanes}, in particular),
but almost always under the assumption that $\bk$ is algebraically closed,
or even $\bk=\Comp$.
In this paper we obtain some partial results valid when $\bk$
is an arbitrary field of characteristic zero.
We are particularly interested in the following subclass:

\begin{definition}
Given a field $\bk$ of characteristic zero,
let $\Dan(\bk)$ be the class of $\bk$-algebras isomorphic to
$ \bk[X,Y,Z] / (XY - \phi(Z)) $
for some nonconstant polynomial in one variable
$\phi(Z) \in \bk[Z] \setminus \bk$,
where $X,Y,Z$ are indeterminates over $\bk$.
\end{definition}

The class $\Dan( \bk )$ was studied in
\cite{Dai:slice1}, \cite{Dai:LSC} and \cite{LML:GpsAutoms}, in particular.
It is well-known that if $R \in \Dan( \bk )$ then $R$ is a $2$-dimensional 
normal affine domain satisfying $\ML(R)=\bk$.
It is also known that the converse is not true,
which raises the following:
\begin{questionn}
\it Suppose that $R$ is a $2$-dimensional affine $\bk$-domain
with $\ML(R)=\bk$.
Under what additional assumptions can we infer that $R \in \Dan(\bk)$?
\end{questionn}
Section~3 completely answers this question
in the case where $R$ is a smooth $\bk$-algebra.  This is achieved by
reducing to the case $\bk=\Comp$, which was solved by
Bandman and Makar-Limanov.
This reduction is non-trivial,
and makes essential use of the main result of Section~2.
Also note Corollary~\ref{NewCorUFD}, which gives a pleasant answer to the
above question in the factorial case.
Then we derive several consequences from Section~3, 
for instance
consider the following special case of Theorem~\ref{dwddwdwdwddwd}:
\begin{quote}
\it Let $R$ be a localization of a polynomial ring in finitely many variables
over a field of characteristic zero. If $\ML(R) = K$ for some field
$K \subset R$ such that $\trdeg_KR=2$, then $R \in \Dan( K )$.
\end{quote}
In turn, this has consequences in the study of $G_a$-actions on $\Comp^n$.

\begin{conventions}
All rings and algebras are commutative, associative and unital.
If $A$ is a ring, we write $A^*$ for the units of $A$; if $A$ is a domain,
$\Frac A$ is its field of fractions.
If $A \subseteq B$ are rings,
``\,$B = A^{[n]}$\,'' means that $B$ is $A$-isomorphic to the polynomial
algebra in $n$ variables over $A$.
If $L/K$ is a field extension,
``\,$L = K^{(n)}$\,'' means that $L$ is a purely transcendental extension
of $K$ and $\trdeg_KL=n$ (transcendence degree).
\end{conventions}

\smallskip
In \cite{Dai:LSC}, one defines a Danielewski surface to be a pair 
$(R,\bk)$ such that $R \in \Dan(\bk)$.
In the present paper we avoid using the term
``Danielewski surface'' in that sense, because it is 
incompatible with accepted usage.
The reader should keep this in mind when consulting \cite{Dai:LSC}
(our main reference for Section~2).

\section{Base extension}

Let $\bk$ be a field of characteristic zero.
It is clear that if $R \in \Dan( \bk )$
then $K \otimes_\bk R \in \Dan( K )$ for every field
extension $K/\bk$.
However, if $K \otimes_\bk R \in \Dan( K )$ for some $K$,
it does not follow that $R \in \Dan( \bk )$
(see Example~\ref{hghgoioioiooiop}, below).

\begin{remark}\label{rururuurytytyru}
{\it If $R \in \Dan( \bk )$
then $\Spec R$ has infinitely many $\bk$-rational points.}
(Indeed, if $R = \bk[X,Y,Z] / (XY - \phi(Z))$ then
there is a bijection between
the set of $\bk$-rational points of $\Spec R$
and the zero-set in $\bk^3$ of the polynomial $XY-\phi(Z)$.)
\end{remark}

\begin{example}\label{hghgoioioiooiop}
Let $A = \Reals[X,Y,Z] / (f)$, where $f = X^2 + Y^2 + Z^2$.
Viewing $f$ as an element of $\Comp[X,Y,Z]$ we have 
$f = (X+iY)(X-iY) + Z^2$ (where $i^2=-1$),
so $\Comp \otimes_\Reals A \isom \Comp[U,V,W] / (UV+W^2) \in \Dan( \Comp )$.
As $\Spec A$ has only one $\Reals$-rational point,
$A \notin \Dan( \Reals )$ by Remark~\ref{rururuurytytyru}.
Thus
$$
\text{$A \notin \Dan( \Reals )$  and  
$\Comp \otimes_\Reals A \in \Dan( \Comp )$.}
$$
Note%
\footnote{A different proof that $\ML(A)=A$ is given
in \cite[9.21]{Freud:Book}.}
that Theorem~\ref{DanML} (below) implies that $\ML(A) = A$.
Moreover, if we define
$A' = \Reals[U,V,W] / (UV+W^2) \in \Dan( \Reals )$
then $A \not\isom A'$
but $\Comp \otimes_\Reals A \isom \Comp \otimes_\Reals A'$.
\end{example}

\begin{theorem}\label{DanML}
For an algebra $R$ over a field $\bk$ of characteristic zero,
the following conditions are equivalent:
\begin{enumerate}

\item[(a)] $R \in \Dan( \bk )$

\item[(b)] $\ML(R) \neq R$ and 
there exists a field extension $K/\bk$ such that
$K \otimes_\bk R \in \Dan( K )$.

\end{enumerate}
\end{theorem}

We shall prove this after some preparation.

\begin{somefacts}\label{fofpofpofpofpof}
Refer to \cite{VDE:book} or \cite{Freud:Book} for background on
locally nilpotent derivations.
Statement \eqref{nnnntttjptjnptjtpjn} is due to Rentschler \cite{Rent}
and \eqref{cvvcccpcpccpcpvcpvv}
to Nouaz\'e and Gabriel~\cite{GabNou} 
and Wright~\cite{Wright:JacConj}. 
\begin{enumerate}
\renewcommand{\theenumi}{\alph{enumi}}

\item\label{vcvcvctgfhyegf}
If $A \in \klnd(B)$ where $B$ is a domain of characteristic zero then 
$A$ is \textit{factorially closed\/} in $B$ (i.e., if $x,y \in B\setminus\{0\}$ and
$xy \in A$ then $x, y \in A$).
It follows that $\ML(B)$ is factorially closed in $B$.
Any factorially closed subring $A$ of $B$ is in particular
\textit{algebraically closed\/} in $B$ (i.e., if $x \in B$ is a root
of a nonzero polynomial with coefficients in $A$ then $x\in A$) and
satisfies $A^*=B^*$
(in particular, any field contained in $B$ is contained in $A$).

\item\label{shhohsohsshhshoohs}
Let $B$ be a noetherian domain of characteristic zero.
If $0 \neq D \in \lnd(B)$
then $D = \alpha D_0$ for some $\alpha \in \ker(D)$ and
$D_0 \in \lnd(B)$ where $D_0$ is \textit{irreducible\/}
(i.e., the only principal ideal of $B$ which contains $D_0(B)$ is $B$).

\item\label{nnnntttjptjnptjtpjn}
Let $B = \kk2$ where $\bk$ is a field of characteristic zero.
If $D \in \lnd(B)$ is irreducible then
there exist $X,Y$ such that $B=\bk[X,Y]$ and 
$D = \partial / \partial Y$.

\item\label{cvvcccpcpccpcpvcpvv}
Let $B$ be a $\Rat$-algebra. If $D \in \lnd(B)$
and $s\in B$ satisfy $Ds\in B^*$ then $B =A[s]=A^{[1]}$ where $A = \ker D$.

\end{enumerate}
\end{somefacts}

\begin{lemma}\label{bvbvbvbryyryyrt}
Let $\bk$ be a field of characteristic zero and $R$ a $\bk$-algebra
satisfying:
$$
\textit{there exists a field extension $\ck/\bk$ such that
$\ck \otimes_\bk R \in \Dan( \ck )$.}
$$
Then $R$ is a two-dimensional normal affine domain over $\bk$
and $R^* = \bk^*$.
\end{lemma}

\begin{proof}
This is rather simple but it will be convenient to refer to this proof later.
Choose a field extension $\ck/\bk$ such that
$\ck \otimes_\bk R \in \Dan( \ck )$ and
let $\bar R = \ck \otimes_\bk R$.
As $R$ is a flat $\bk$-module, the canonical homomorphism
$\bk \otimes_\bk R \to \ck \otimes_\bk R$ is injective,
so we may regard $R$ as a subring of $\bar R$.
In particular, $R$ is an integral domain and we have the diagram:
\begin{equation*} 
\raisebox{6mm}{\xymatrix{\ck\  \ar @{^(->}[r] & \bar R\  \ar @{^(->}[r] 
& S^{-1}\bar R \ar @{^(->}[r] &   \Frac \bar R \\
\bk\ \ar @{^(->}[r] \ar @{^(->}[u] &  R\  \ar @{^(->}[r] \ar @{^(->}[u]  
&   \Frac R \ar @{^(->}[u]  }}
\end{equation*}
where $S = R \setminus \{0\}$.
Let $\Beul$ be a basis of $\ck$ over $\bk$ such that $1 \in \Beul$.
Note that $\Beul$ is also a basis of the free $R$-module $\bar R$
and of the vector space $S^{-1} \bar R$ over $\Frac R$.
It follows:
\begin{equation}\label{kdfjwehfui}
\ck \cap R = \bk  \quad \textit{and} \quad \bar R \cap \Frac R = R .
\end{equation}
As $\bar R \in \Dan( \ck )$,  \cite[2.3]{Dai:LSC} implies that
$\bar R^* = \ck^*$ and that $\bar R$ is a normal domain;
so \eqref{kdfjwehfui} implies that $R^* = \bk^*$ and that $R$
is a normal domain.  Also:
\begin{equation}\label{xzsqeasfqwcvsoioi}
\textit{If $E$ is a subset of $R$ such that $\ck[ E ] = \bar R$,
then $\bk[E] = R$.}
\end{equation}
Indeed, $\Beul$ is a basis of the $R$-module $\bar R$
and a spanning set of the $\bk[ E ]$-module $\bar R$;
as $\bk[E] \subseteq R$, it follows that $\bk[E] = R$.

Note that $R$ is affine over $\bk$, 
by \eqref{xzsqeasfqwcvsoioi} and the fact that 
$\bar R$ is affine over $\ck$.
Let $n = \dim R$ then, by Noether Normalization Lemma,
there exists a subalgebra $R_0 = \kk n$ of $R$ over which $R$ is integral.
Then $\bar R = \ck \otimes_\bk R$ is integral over
$\ck \otimes_\bk R_0 = \ck^{[n]}$, so $n = \dim \bar R = 2$.
\end{proof}

We borrow the following notation from \cite[2.1]{Dai:LSC}.

\begin{definition} \label{danielewski}
Given a $\bk$-algebra $R$, let
$\Gamma_\bk(R)$ denote the (possibly empty) set of ordered triples
$(x_1,x_2,y)\in R\times R\times R$ satisfying:
\begin{quote}\it
The $\bk$-homomorphism $\bk[X_1,X_2,Y] \to R$ defined by
$$
\text{$X_1\mapsto x_1$, $X_2\mapsto x_2$ and $Y\mapsto y$}
$$
is surjective and has kernel equal to $(X_1X_2 - \phi(Y))\bk[X_1,X_2,Y]$
for some nonconstant polynomial in one variable $\phi(Y)\in\bk[Y]$.
\end{quote}
Note that $R \in \Dan( \bk )$ if and only if $\Gamma_\bk(R)\neq\emptyset$.
\end{definition}

\begin{proof}[Proof of Theorem~\ref{DanML}]
That $R \in \Dan( \bk )$ implies $\ML(R) = \bk$
is well-known (for instance it follows from part~(d) of \cite[2.3]{Dai:LSC}),
so it suffices to prove that (b) implies (a).

Suppose that $R$ satisfies (b).
Note that if $K/\bk$ is a field extension satisfying
$K \otimes_\bk R \in \Dan( K )$ then for any field extension $L/K$ we have
$L \otimes_\bk R \in \Dan( L )$.
In particular, there exists a field extension $\ck / \bk$ such that
$\ck \otimes_\bk R \in \Dan( \ck )$ and such that $\ck$ is an algebraically
closed field. We fix such a field $\ck$.
The fact that $\ck$ is algebraically closed implies that
\begin{equation}\label{FixedField}
\text{the fixed field $\ck^G$ is equal to $\bk$}
\end{equation}
where $G = \Gal( \ck / \bk )$.
We use the notation ($\bar R$, $\Beul$, etc) introduced
in the proof of Lemma~\ref{bvbvbvbryyryyrt}.
As $\ML(R) \neq R$, there exists $0 \neq D \in \lnd(R)$.
Let $\bar D \in \lnd( \bar R)$ be the unique extension of $D$,
let $A = \ker D$ and $\bar A = \ker \bar D$.

It follows from \cite{Dai:LSC} that $\bar A = \ck^{[1]}$
(\cite[2.3]{Dai:LSC} shows that
some element of $\klnd( \bar R )$ is a $\ck^{[1]}$ and,
by \cite[2.7.2]{Dai:LSC},
$\Aut_\ck( \bar R )$ acts transitively on $\klnd( \bar R )$).
Applying the exact functor $\ck \otimes_\bk\,\underline{\ \ }$
to the exact sequence $0 \to A \to R \xrightarrow{D} R$ of $\bk$-linear
maps shows that $\ck \otimes_\bk A = \bar A = \ck^{[1]}$,
so $A = \kk1$.
Choose $f \in R$ such that $A = \bk[f]$, then $\bar A = \ck[f]$.

Consider the nonzero ideals
$I = A \cap D(R)$
and $\bar I = \bar A \cap \bar D(\bar R)$
of $A$ and $\bar A$ respectively.
Let $\psi \in A$ and $s\in R$ be such that $I = \psi A$
and $D(s) = \psi$.
We claim that 
\begin{equation}\label{vcvcvtrtrtrtey}
\bar I = \psi \bar A .
\end{equation}
Indeed, an arbitrary element of $\bar I$ is of the form
$\bar D( \sigma )$ where $\sigma \in \bar R$ and $\bar D^2( \sigma )=0$.
Write $\sigma = \sum_{\lambda \in \Beul} s_\lambda \, \lambda$
with $s_\lambda \in R$,
then  $0 = \bar D^2( \sigma )
= \sum_{\lambda \in \Beul} D^2( s_\lambda ) \, \lambda$,
so  for all $\lambda \in \Beul$ we have $D^2( s_\lambda ) = 0$,
hence $D( s_\lambda ) \in I = \psi A$, and consequently
$\bar D( \sigma ) \in \psi \bar A$, which proves~\eqref{vcvcvtrtrtrtey}.

By \ref{fofpofpofpofpof}(\ref{shhohsohsshhshoohs}),
$\bar D = \alpha \Delta$
for some $\alpha \in \bar A \setminus \{0\}$ and some 
irreducible $\Delta \in \lnd(\bar R)$.
Consider the nonzero ideal $I_0 = \bar A \cap \Delta( \bar R )$ of $\bar A$.
We claim that 
\begin{equation}\label{vwvwvwvwvvwfe}
I_0 = \Delta(s) \bar A .
\end{equation}
To see this, consider an arbitrary element $\Delta( \sigma )$ of $I_0$
(where $\sigma \in \bar R$, $\Delta^2( \sigma ) = 0$).
Then
$ \alpha \Delta( \sigma ) = \bar D( \sigma ) \in \bar I = \psi \bar A
= \bar D(s) \bar A = \alpha \Delta(s) \bar A $,
so $\Delta( \sigma ) \in  \Delta(s) \bar A$ and \eqref{vwvwvwvwvvwfe} is
proved.

Consider the case where $\Delta(s) \in \bar R^*$.
Then $\bar R = \bar A[s] = \ck[ f, s ]$
by \ref{fofpofpofpofpof}(\ref{cvvcccpcpccpcpvcpvv}),
so \eqref{xzsqeasfqwcvsoioi} implies that $R = \bk[f,s] = \kk2$,
so in particular $R \in \Dan(\bk)$ and we are done.

From now-on assume that $\Delta(s) \not\in\bar R^*$.
By \cite[2.8]{Dai:LSC},
$\bar A = \ck[ \Delta(y) ]$ for some $y \in \bar R$.
Note that $\Delta(y) \in I_0$, so \eqref{vwvwvwvwvvwfe} gives
$\Delta(s) \mid \Delta(y)$ in $\bar A$.
As $\Delta(y)$ is an irreducible element of $\bar A$
(because $\ck[ \Delta(y) ] = \bar A = \ck^{[1]}$) and 
$\Delta(s) \not\in\bar A^*$,
we have $\ck[ \Delta(s) ] = \bar A = \ck[f]$ and
consequently $\Delta(s) = \mu(f- \lambda)$ for some $\mu\in\ck^*$,
$\lambda \in \ck$.
We may as well replace $\Delta$ by $\mu^{-1}\Delta$, so 
\begin{equation}\label{hggfdewqqqq}
\Delta(s) = f - \lambda, \quad \textit{for some $\lambda \in \ck$}.
\end{equation}
We claim:
\begin{equation}\label{vcvcvtrttryey}
\setspec{ c \in \ck }
{ \text{$\bar R / (f-c)\bar R$ is not an integral domain} }
= \{ \lambda \} .
\end{equation}
Indeed, \cite[2.8]{Dai:LSC} implies that there exists $x_2 \in \bar R$
such that $( f - \lambda, x_2, s) \in \Gamma_{\ck}( \bar R )$.
This means (cf.\ \ref{danielewski}) that the $\ck$-homomorphism
$\pi : \ck[X_1, X_2, Y] \to \bar R$ defined by
$X_1 \mapsto f-\lambda$,
$X_2 \mapsto x_2$,
$Y \mapsto s$,
is surjective and has kernel $(X_1X_2 - P(Y))$ for some nonconstant
$P(Y) \in \ck[Y]$ (where $X_1, X_2, Y$ are indeterminates).
By \eqref{vwvwvwvwvvwfe} and $\Delta(s) \not\in\bar R^*$, we see
that there does not exist $\sigma \in \bar R$ such that $\Delta(\sigma)=1$;
as $\Delta$ is irreducible,
it follows from \ref{fofpofpofpofpof}(c)
that $\bar R \neq \ck^{[2]}$ and hence that $\deg_YP(Y) > 1$.
Thus, for $c \in \ck$,
$$
\bar R / (f - c) \bar R
\ \isom \ 
\ck[X_1, X_2, Y] / ( X_1 - (c-\lambda),\ X_1X_2 - P(Y) )
$$
is a domain if and only if $c \neq \lambda$.
This proves~\eqref{vcvcvtrttryey}.

Let $\theta \in \Gal( \ck / \bk )$.
Then $\theta$ extends to some $\Theta \in \Aut_R( \bar R )$ and
$\Theta$ determines a ring isomorphism
$$
\bar R / (f - \lambda) \bar R \isom
\bar R / \Theta(f - \lambda) \bar R
= \bar R / (f - \theta(\lambda)) \bar R.
$$
So $\bar R / (f - \theta(\lambda)) \bar R$ is not a domain and
it follows from~\eqref{vcvcvtrttryey} that $\theta(\lambda)=\lambda$.
As this holds for every $\theta \in \Gal( \ck / \bk )$,
\eqref{FixedField} implies that $\lambda \in \bk$.
To summarize, if we define $x_1 = f-\lambda$ then
$$
\textit{
$x_1, s \in R$ and there exists $x_2 \in \bar R$ such that 
$(x_1, x_2, s) \in \Gamma_\ck( \bar R )$.}
$$
We now show that $x_2$ can be chosen in $R$.
Consider the ideals
$J = \bk[s] \cap x_1 R$
of $\bk[s]$ and
$\bar J = \ck[s] \cap x_1 \bar R$
of $\ck[s]$,
and choose $\phi(Y) \in \bk[Y]$ such that $J = \phi(s) \bk[s]$.
Let $\Phi(s)$ be any element of $\bar J$ (where $\Phi(Y) \in \ck[Y]$).
Then $\Phi(s) = x_1 G$ for some $G \in \bar R$.
As $\Beul$ is a basis
of the $R$-module $\bar R$ and also of the $\bk[Y]$-module $\ck[Y]$,
we may write $G = \sum_{ \lambda \in \Beul } G_\lambda \lambda$
(where $G_\lambda \in R$) and
$\Phi = \sum_{ \lambda \in \Beul } \Phi_\lambda \lambda$
(where $\Phi_\lambda  \in \bk[Y]$).
Then
$\sum_{ \lambda \in \Beul } (x_1 G_\lambda) \lambda
= \Phi(s) = 
\sum_{ \lambda \in \Beul } \Phi_\lambda(s) \lambda$,
so for every $\lambda \in \Beul$ we have
$\Phi_\lambda(s) = x_1 G_\lambda$, i.e., 
$\Phi_\lambda(s) \in J = \phi(s) \bk[s]$.
We obtain that $\Phi(s) \in \phi(s) \ck[s]$, so:
$$
\bar J = \phi(s) \ck[s] .
$$
On the other hand, \cite[2.4]{Dai:LSC} asserts that $\bar J = x_1 x_2 \ck[s]$,
so $x_1 x_2 = \mu \phi(s)$ for some $\mu \in \ck^*$.
It is clear that if $(x_1, x_2, s)$ belongs to $\Gamma_\ck( \bar R )$
then so does $(x_1, \mu^{-1} x_2, s)$;
so there exists $x_2 \in \bar R$ such that
$(x_1, x_2, s) \in \Gamma_\ck( \bar R )$
and $x_1 x_2 = \phi(s)$.
As $x_2 = \phi(s)/x_1 \in \Frac R$, \eqref{kdfjwehfui} implies
that $x_2 \in R$. Thus
$$
\textit{$(x_1, x_2, s) \in \Gamma_\ck( \bar R )$, where $x_1, x_2, s \in R$.}
$$
In particular we have $\bar R = \ck[ x_1, x_2, s ]$,
so \eqref{xzsqeasfqwcvsoioi} gives $R = \bk[ x_1, x_2, s ]$.
As $x_1 x_2 = \phi(s)$ where $\phi(Y) \in \bk[Y]$ is nonconstant,
it follows that 
$(x_1, x_2, s) \in \Gamma_\bk( R )$ and hence that $R \in \Dan( \bk )$.
\end{proof}

\section{On a result of Bandman and Makar-Limanov}
\label{Sec:BandML}

In this paper we adopt the following:

\begin{definition}
Let $R$ be an affine algebra over a field $\bk$ and let $q = \dim R$.
We say that $R$ is a \textit{complete intersection over $\bk$}
if $R \isom \bk[X_1, \dots, X_{p+q} ] / ( f_1, \dots, f_p )$
for some $p\ge0$ and some 
$ f_1, \dots, f_p  \in \bk[X_1, \dots, X_{p+q} ]$.
\end{definition}

We refer to \cite[28.D]{Matsumura}
for the definition of a \textit{smooth $\bk$-algebra}
and to \cite[26.C]{Matsumura} for the definition of the
$R$-module $\Omega_{R/\bk}$
(the module of differentials of $R$ over $\bk$), where $R$ is 
a $\bk$-algebra.

\begin{theorem}\label{cxcxcxxexexeexexr}
Let $\bk$ be a field of characteristic zero and $R$ a smooth
affine $\bk$-domain of dimension $2$ such that $\ML(R) = \bk$.
Then the following are equivalent:
\begin{enumerate}

\item[(a)] $R \in \Dan( \bk )$

\item[(b)] $R$ is generated by $3$ elements as a $\bk$-algebra

\item[(c)] $R$ is a complete intersection over $\bk$

\item[(d)] $\bigwedge^2 \Omega_{R/\bk} \isom R$.

\end{enumerate}
\end{theorem}

We shall prove this by reducing to the case $\bk=\Comp$,
which was proved by Bandman and Makar-Limanov in \cite{BandML:AffSurfAK}.
That reduction makes essential use of Theorem~\ref{DanML}.

\begin{remark}
Let $\bk$ be a field of characteristic zero.
According to the definition of ``Danielewski surface over $\bk$''
given in \cite{Dub:EmbeddOfDans}, 
one has the following situation:
$$
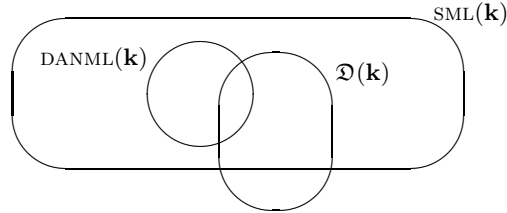
\setlength{\unitlength}{1mm}
\begin{picture}(65,28)(-5,-5)
\put(25,10){\oval(60,20)}
\put(20,10){\circle{20}}
\put(30,5){\oval(15,21)}
\put(13,13){\makebox(0,0)[br]{\scriptsize $\danml(\bk)$}}
\put(38,11){\makebox(0,0)[bl]{\scriptsize $\Dan(\bk)$}}
\put(51,19){\makebox(0,0)[bl]{\scriptsize $\sml(\bk)$}}
\end{picture}
$$
where $\danml(\bk)$ is the class of Danielewski surfaces $S$ over
$\bk$ satisfying $\ML(S)=\bk$,  $\sml(\bk)$ is the larger class
of smooth affine surfaces $S$ over $\bk$ satisfying $\ML(S)=\bk$,
and $\Dan(\bk)$ is the class of surfaces corresponding to the
already defined class $\Dan(\bk)$ of $\bk$-algebras.
Among other things, paper~\cite{Dub:EmbeddOfDans}
classifies the elements of $\danml(\bk)$ and
characterizes those which belong to $\Dan(\bk)$.
In contrast, Theorem~\ref{cxcxcxxexexeexexr} characterizes the
elements of $\sml(\bk)$ which belong to $\Dan(\bk)$.
\end{remark}

\begin{remark}\label{RemCanSheaf}
Let $R$ be a $q$-dimensional smooth affine domain over a field $\bk$ of
characteristic zero.
Then $X = \Spec R$ is in particular an irreducible regular scheme of
finite type over the perfect field $\bk$;
so, by \cite[ex.~8.1(c), p.\ 187]{Hartshorne}, the sheaf of differentials
$\Omega_{X/\bk}$ is locally free of rank $q$;
so the canonical sheaf $\omega_X = \bigwedge^q \Omega_{X/\bk}$
is locally free of rank $1$, i.e., is an invertible sheaf on $X$.
As $\omega_X$ and the structure sheaf $\Oeul_X$ are respectively
the sheaves associated
to the $R$-modules $\bigwedge^q \Omega_{R/\bk}$ and $R$,
the condition $\bigwedge^q \Omega_{R/\bk} \isom R$ is equivalent
to $\omega_X \isom \Oeul_X$ (one says that $X$ has trivial canonical sheaf).
This is also equivalent to the canonical divisor of $X$ being linearly
equivalent to zero (because $\Pic(X) \isom \Cl(X)$ by
\cite[6.16 p.\ 145]{Hartshorne}).
\end{remark}

\begin{remark}\label{fjefkwhehjwgjsklfjajhh}
Let $A'$ and $B$ be algebras over a ring $A$ and let $B' = A' \otimes_A B$.
Then $\Omega_{B'/A'} \isom B'\otimes_B \Omega_{B/A}$
(cf.\ \cite[p.\ 186]{Matsumura}) and, for any $B$-module $M$,
$\bigwedge^n ( B'\otimes_B M ) \isom 
B'\otimes_B \bigwedge^n M$ for every $n$
(\cite{BourbakiAlgI_III}, Chap.~3, \S\,7, No~5, Prop.~8).
Consequently,
$\bigwedge^n \Omega_{B'/A'} \isom B'\otimes_B \bigwedge^n \Omega_{B/A}$.
\end{remark}

\begin{lemma}\label{fjejkhfjheudgygygwy}
{\it Let $R$ be an algebra over a field $\bk$.
If $R$ is a complete intersection over $\bk$ and a smooth $\bk$-algebra,
then $\bigwedge^q \Omega_{R/\bk} \isom R$ where $q = \dim R$.}
\end{lemma}

This is the well-known fact that a smooth complete intersection
has trivial canonical sheaf, but we don't know a suitable reference
so we sketch a proof.

\begin{proof}[Proof of \ref{fjejkhfjheudgygygwy}]
Let $R = \bk[X_1, \dots, X_{p+q} ] / (f_1, \dots, f_p)$ and
let $\phi_{ij} \in R$ be the image of $\frac{ \partial f_j }{ \partial X_i }$.
Because $R$ is smooth over $\bk$, \cite[29.E]{Matsumura} implies that the
matrix $(\phi_{ij})$ satisfies:
\begin{equation}\label{fgfigigfgfgfig}
\text{the $p \times p$ determinants of $(\phi_{ij})$ generate
the unit ideal of $R$.}
\end{equation}
By \cite[8.4A, p.~173]{Hartshorne}, there is an exact sequence
\mbox{$ R^p \xrightarrow{\ \phi\ } R^{p+q} \to \Omega_{R/\bk} \to 0 $}
of $R$-linear maps where $\phi$ is the map corresponding to the
matrix $(\phi_{ij})$.
Now if $R$ is a ring and 
$R^p \xrightarrow{\ \phi\ } R^{p+q} \to M \to 0$
is an exact sequence of $R$-linear maps such that $\phi$
satisfies \eqref{fgfigigfgfgfig}, then $\bigwedge^q M \isom R$.
\end{proof}

\begin{lemma}\label{jfewjfkjhlshf;ij;}
Let $R$ be an integral domain containing a field $\bk$ of characteristic zero.
If $R$ is normal and  $\ML(R) = \bk$,
then for any field extension $K$ of $\bk$ we have:
\begin{enumerate}

\item[(a)] $K \otimes_\bk R$ is an integral domain

\item[(b)] $\ML( K \otimes_\bk R ) = K$.

\end{enumerate}
\end{lemma}

\begin{proof}
As $\bk=\ML(R)$ is algebraically closed in $R$
(\ref{fofpofpofpofpof}(\ref{vcvcvctgfhyegf})) and $R$ is normal,
it follows that $\bk$ is algebraically closed in $L = \Frac R$.
By \cite[Cor.~2, p.~198]{ZarSamI}, $K \otimes_\bk L$ is an integral domain.
As $K$ is flat over $\bk$ and $R \to L$ is injective,
$K \otimes_\bk R \to K \otimes_\bk L$ is injective and (a) is proved.

Let $\xi \in \ML( K \otimes_\bk R )$.
Consider a basis $\Beul$ of $K$ over $\bk$;
note that $\Beul$ is also a basis of the free $R$-module
$R' = K \otimes_\bk R$ and write
$\xi = \sum_{\lambda \in \Beul} x_\lambda \lambda$ (where $x_\lambda \in R$).
If $D \in \lnd(R)$ then $D$ extends to an element $D' \in \lnd( R' )$ and
the equation
$ 0 = D' ( \xi ) =  \sum_{\lambda \in \Beul} D(x_\lambda) \lambda $
shows that $D( x_\lambda ) = 0$ for all $\lambda \in \Beul$.
As this holds for every $D \in \lnd(R)$, we have
$x_\lambda \in \ML( R ) = \bk$ for all $\lambda$, so $\xi \in K$.
\end{proof}

\begin{proof}[Proof of Theorem~\ref{cxcxcxxexexeexexr}]
Implications $\text{(a)} \Rightarrow \text{(b)} \Rightarrow \text{(c)}$
are trivial and $\text{(c)} \Rightarrow \text{(d)}$
is Lemma~\ref{fjejkhfjheudgygygwy},
so only $\text{(d)} \Rightarrow \text{(a)}$ requires a proof.
Assume for a moment that $\bk=\Comp$ and suppose that $R$ satisfies~(d).
Then Lemmas~4 and 5 of \cite{BandML:AffSurfAK} imply that $R\in\Dan(\Comp)$,
so the Theorem is valid in the case $\bk=\Comp$.

Let $\bk$ be a field of characteristic zero,
consider a smooth affine $\bk$-domain $R$ of dimension $2$
such that $\ML(R) = \bk$, and suppose that $R$ satisfies~(d).

We have $R \isom \bk[X_1, \dots, X_n ] / ( f_1, \dots, f_m )$
for some $m,n\ge0$ and some $f_1, \dots, f_m \in \bk[X_1, \dots, X_n ]$.
Also consider $D_1, D_2 \in \lnd(R)$ such that $\ker D_1 \cap \ker D_2 = \bk$.
Each $D_i$ can be lifted to a (not necessarely locally nilpotent)
$\bk$-derivation $\delta_i$ of $\bk[X_1, \dots, X_n ]$.
Let $\bk_0$ be a subfield of $\bk$ which is finitely generated over $\Rat$
and which contains all coefficients of the polynomials
$f_i$ and $\delta_i(X_j)$.
Define $R_0 = \bk_0[X_1, \dots, X_n ] / ( f_1, \dots, f_m )$ and
note that $\bk \otimes_{\bk_0} R_0 \isom R$.
As $\bk_0 \to \bk$ is injective and $R_0$ is flat over $\bk_0$,
$\bk_0 \otimes_{\bk_0} R_0 \to \bk \otimes_{\bk_0} R_0$ is injective
and we may regard $R_0$ as a subring of $R$. In particular, $R_0$ is a domain
(a $2$-dimensional affine $\bk_0$-domain).
Also note that $D_i(R_0) \subseteq R_0$ for $i=1,2$;
if $d_i : R_0 \to R_0$ is the restriction of $D_i$
then $d_1, d_2 \in \lnd( R_0 )$ and
$\ker d_1 \cap \ker d_2 = \bk \cap R_0 = \bk_0$
(see \eqref{kdfjwehfui} for the last equality),
showing that $\ML( R_0 ) = \bk_0$.
As $\bk_0$ is a field and $\bk \to R$ is obtained from 
$\bk_0 \to R_0$ by base extension, the fact that
$\bk \to R$ is smooth implies that $\bk_0 \to R_0$ is smooth
(cf.\ \cite[28.O]{Matsumura}).

Consider the $R$-module $M = \bigwedge^2 \Omega_{R/\bk}$
and the $R_0$-module $M_0 = \bigwedge^2 \Omega_{R_0/\bk_0}$.
Consider an isomorphism of $R$-modules
$\theta : R \to M$ and let $\omega = \theta(1)$.
We have $R \otimes_{ R_0 } M_0 \isom M$ by \ref{fjefkwhehjwgjsklfjajhh},
so there is a natural homomorphism
$M_0 \to R \otimes_{ R_0 } M_0 \isom M$, $x \mapsto 1 \otimes x$;
by adjoining a finite subset of $\bk$ to $\bk_0$,
we may arrange that there exists
$\omega_0 \in M_0$ such that $1 \otimes \omega_0 = \omega$.
Consider the $R_0$-linear map $f : R_0 \to M_0$, $f(a)=a\omega_0$.
Note that $R = \bk \otimes_{\bk_0} R_0$ is faithfully flat as an
$R_0$-module and that applying the functor $R \otimes_{R_0}\underline{\ \ }$
to $f$ yields the isomorphism $\theta$; so $f$ is an isomorphism,
so $\bigwedge^2 \Omega_{R_0/\bk_0} \isom R_0$.
As $R \in \Deul(\bk)$ would follow from $R_0 \in \Deul(\bk_0)$,
the problem reduces to proving the case $\bk = \bk_0$ of the theorem.
Now $\bk_0$ is isomorphic to a subfield of $\Comp$, so 
it suffices to prove the theorem in the case $\bk \subseteq \Comp$.

Assume that $\bk \subseteq \Comp$.
As $R$ is smooth over $\bk$,
the local ring $R_\pgoth$ is regular for every $\pgoth \in \Spec R$
(by \cite[28.E,F,K]{Matsumura}) so in particular $R$ is a normal domain.
Then it follows from \ref{jfewjfkjhlshf;ij;} that
$R' = \Comp \otimes_\bk R$ is an integral domain and that $\ML( R' ) = \Comp$.
By \cite[28.G]{Matsumura}, $R'$ is smooth over $\Comp$.
It is clear that $\dim R' = 2$
(for instance see the proof of \ref{bvbvbvbryyryyrt})
and \ref{fjefkwhehjwgjsklfjajhh} gives
$ \bigwedge^2 \Omega_{ R' / \Comp }
\isom
R' \otimes_{ R } \bigwedge^2 \Omega_{ R / \bk }
\isom
R' \otimes_{ R } R
\isom
R'$.
As the Theorem is valid over $\Comp$,
it follows that $R' \in \Deul( \Comp )$.
As $\ML(R) = \bk \neq R$, Theorem~\ref{DanML}
implies that $R \in \Deul( \bk )$.
\end{proof}

\begin{corollary}\label{NewCorUFD}
Let $R$ be a $2$-dimensional affine domain over a field $\bk$ of
characteristic zero.
If $R$ is a UFD and a smooth $\bk$-algebra satisfying
$\ML(R) = \bk$, then $R \in \Dan( \bk )$.
\end{corollary}

\begin{proof}
Since $R$ is a UFD,
the scheme $X = \Spec R$ has a trivial divisor class group
\cite[6.2 p.\ 131]{Hartshorne}.
By Remark~\ref{RemCanSheaf}, it follows that 
$\bigwedge^2 \Omega_{R/\bk} \isom R$
and the desired conclusion follows from Theorem~\ref{cxcxcxxexexeexexr}.
\end{proof}

\section{Localizations of nice rings}
\label{dkfj;wejij2wlrhwuiehl;wjeioqje;}

Throughout this section we fix a field $\bk$ of characteristic zero
and we consider the class $\Neul(\bk)$ of $\bk$-algebras $B$
satisfying the following conditions:
\begin{quote}\it
$B$ is a geometrically integral affine $\bk$-domain which is
smooth over $\bk$ and satisfies at least one of the following conditions:
\begin{itemize}
\item $B$ is a UFD; or
\item $B$ is a complete intersection over $\bk$.
\end{itemize}
\end{quote}
Note that $\kk n \in \Neul(\bk)$ for every $n$.

\begin{theorem}\label{dwddwdwdwddwd}
Suppose that $R$ is a localization of a ring belonging to the
class $\Neul(\bk)$.
If $\ML(R) = K$ for some field
$K \subset R$ such that $\trdeg_KR=2$, then $R \in \Dan( K )$.
\end{theorem}

\begin{lemma}\label{nbnbnbniuiuioypw}
Let $B \in \Neul( \bk )$,
let $E$ be a finitely generated $\bk$-subalbebra of $B$ and
let $S = E \setminus\{0\}$.
Then $S^{-1}B$ is a smooth algebra over the field $S^{-1}E$.
\end{lemma}

\begin{proof}
Let $\ck$ be an algebraic closure of $\bk$
and define $\bar E = \ck \otimes_\bk E$
and $\bar B = \ck \otimes_\bk B$.
Note that $\bar B$ is a domain because $B$ is geometrically integral,
and $\bar E \to \bar B$ is injective because $\ck$ is flat over $\bk$.
Let $K = \Frac E$ and $L = \Frac \bar E$.
As $\bar B$ is smooth over $\ck$,
applying \cite[10.7, p.~272]{Hartshorne} to $\Spec\bar B \to \Spec \bar E$
implies that $L \to L \otimes_{\bar E} \bar B$ is smooth.
It is not difficult to see
that $L \to L \otimes_{\bar E} \bar B$ 
is obtained from  $K \to K \otimes_{E} B$ by base extension.
As $K$ is a field and $L \to L \otimes_{\bar E} \bar B$ is smooth,
it follows from \cite[28.O]{Matsumura} that $K \to K \otimes_{E} B$ is smooth.
\end{proof}

\begin{lemma}\label{fwkefkjnfkja;kl}
Let $B \in \Neul( \bk )$, let $S$ be a multiplicative subset of  $B$
and suppose that $K$ is a field such that
$\bk \cup S \subseteq K \subseteq S^{-1}B$.
Then $S^{-1}B$ is a smooth $K$-algebra and some
transcendence basis of $K/\bk$ is a subset of $B$.
\end{lemma}

\begin{proof}
Note that $K/\bk$ is a finitely generated field extension and write 
$K = \bk( \alpha_1, \dots, \alpha_m )$.
For each $i$ we have $\alpha_i = b_i/s_i$ for some $b_i \in B$ and 
$s_i \in S$; as $S \subseteq K$, we have $b_i = s_i \alpha_i \in K$.
Define $E = \bk[b_1, \dots, b_m, s_1, \dots, s_m] \subseteq K$ and
$S_1 = E \setminus \{0\}$, then
$S_1^{-1} E = K$ and hence $S_1^{-1} B = S^{-1} B$.
By Lemma~\ref{nbnbnbniuiuioypw}, $S^{-1} B$ is a smooth $K$-algebra.
Moreover, $\{ b_1, \dots, b_m, s_1, \dots, s_m \}$
contains a transcendence basis of $K/\bk$.
\end{proof}

\begin{proof}[Proof of Theorem~\ref{dwddwdwdwddwd}]
We have $R = S^{-1}B$ for some $B \in \Neul( \bk )$ and some multiplicative
subset $S$ of $B$.
As $\bk^* \cup S \subseteq R^* \subseteq \ML(R) = K$,
$R$ is smooth over $K$ by Lemma~\ref{fwkefkjnfkja;kl}.
By definition of $\Neul( \bk )$, $B$ is a UFD or a complete intersection
over $\bk$.

If $B$ is a UFD then so is $R$;
in this case we obtain $R \in \Dan( K )$ by Corollary~\ref{NewCorUFD},
so we are done.

From now-on, assume that $B$ is a complete intersection over $\bk$.
Let $q=\dim B$ and write $B = \bk[X_1, \dots, X_{p+q}] / (G_1, \dots, G_p)$.
Using Lemma~\ref{fwkefkjnfkja;kl} again,
choose a transcendence basis $\{ f_1, \dots, f_{q-2} \}$ of $K$ over
$\bk$ such that $f_1, \dots, f_{q-2} \in B$;
let $S_0 = \bk[ f_1, \dots, f_{q-2} ] \setminus \{ 0 \}$ and
$K_0 = \bk( f_1, \dots, f_{q-2} )$.
We claim:
\begin{equation}\label{fjlwehfhwlehj;wkej;qilui}
\text{$S_0^{-1}B$ is a complete intersection over $K_0$.}
\end{equation}
Let us prove this.
For $1 \le i \le q-2$, choose $F_i \in \bk[X_1, \dots, X_{p+q}]$
such that $\pi( F_i ) = f_i$ where
$\pi : \bk[X_1, \dots, X_{p+q}] \to B$ is the canonical epimorphism.
Also, let $T_1, \dots, T_{q-2}$ be extra indeterminates.
The $\bk$-homomorphism
$\bk[T_1, \dots, T_{q-2}, X_1, \dots, X_{p+q} ] \to B$
which maps $T_i$ to $f_i$ and $X_i$ to $\pi(X_i)$ has kernel
$(G_1, \dots, G_p, F_1-T_1, \dots, F_{q-2} - T_{q-2} )$,
so there is an isomorphism of $\bk$-algebras
$$
B \isom \bk[T_1, \dots, T_{q-2}, X_1, \dots, X_{p+q} ]
/ (G_1, \dots, G_p, F_1-T_1, \dots, F_{q-2} - T_{q-2} ) .
$$
Localization gives an an isomorphism of $\bk$-algebras
\begin{equation}\label{kfjwjehfiwuelj;;weklj}
S_0^{-1}B \isom
\bk(T_1, \dots, T_{q-2}) [ X_1, \dots, X_{p+q} ]
/ (G_1, \dots, G_p, F_1-T_1, \dots, F_{q-2} - T_{q-2} )
\end{equation}
which maps $K_0$ onto $\bk(T_1, \dots, T_{q-2})$.
As the right hand side of \eqref{kfjwjehfiwuelj;;weklj} is 
a complete intersection over $\bk(T_1, \dots, T_{q-2})$,
assertion \eqref{fjlwehfhwlehj;wkej;qilui} is proved.
Then we obtain
\begin{equation}\label{bvbbvbvytytytyru}
\textstyle
\bigwedge^2 \Omega_{S_0^{-1}B/K_0} \isom S_0^{-1}B
\end{equation}
by Lemma~\ref{fjejkhfjheudgygygwy},
because $S_0^{-1}B$ is a smooth $K_0$-algebra by Lemma~\ref{nbnbnbniuiuioypw}.

Each element of $K$ belongs to $\Frac( S_0^{-1}B )$ and is
algebraic over $K_0$, hence integral over $S_0^{-1}B$;
as $S_0^{-1}B$ is normal, $K \subseteq S_0^{-1}B$ and hence $S_0^{-1}B = R$.
We may therefore rewrite \eqref{bvbbvbvytytytyru} as:
\begin{equation}\label{fkefjjkhhjghgffsdeq}
\textstyle
\bigwedge^2 \Omega_{R/K_0} \isom R .
\end{equation}
Applying \cite[26.H]{Matsumura}
to $K_0 \subseteq K \subseteq R$ gives the exact sequence of $R$-modules
$$
\Omega_{K/K_0} \otimes_K R \to \Omega_{R/K_0} \to  \Omega_{R/K} \to 0,
$$
where $\Omega_{K/K_0} = 0$ by \cite[27.B]{Matsumura}.
So $\Omega_{R/K} \isom \Omega_{R/K_0}$ and hence
\eqref{fkefjjkhhjghgffsdeq} gives $\bigwedge^2 \Omega_{R/K} \isom R$.
So $R \in \Dan( K )$ by Theorem~\ref{cxcxcxxexexeexexr}.
\end{proof}


Let $\bk$ be a field of characteristic zero,
let $B \in \Neul( \bk )$ and consider locally nilpotent derivations
$D: B \to B$.
See \ref{fefefefeer} for the definition of $\klnd(B)$.
It is known that if $A \in \klnd(B)$ then $\trdeg_A(B)=1$,
and if $A_1, A_2$ are distinct elements of $\klnd(B)$
then $\trdeg_{ A_1 \cap A_2 }(B) \ge 2$.
We are interested in the situation where 
$\trdeg_{ A_1 \cap A_2 }(B) = 2$, i.e., when 
$A_1, A_2$ are distinct and have an intersetion which is as large as 
possible.

\begin{corollary}\label{ofoofofoofofoo}
Let $B \in \Neul( \bk )$, where $\bk$ is a field of characteristic zero.
If $A_1, A_2 \in \klnd(B)$ are such that
$\trdeg_{ A_1 \cap A_2 }(B) = 2$, then the following hold.
\begin{enumerate}

\item[(a)] Let $R = A_1 \cap A_2$ and $K = \Frac R$.
Then $K \otimes_R B \in \Dan( K )$.

\item[(b)] If $B$ is a UFD then there exists a finite sequence
of local slice constructions which transforms $A_1$ into $A_2$.

\end{enumerate}
\end{corollary}

\begin{remark*}
This generalizes results~1.10 and 1.13 of \cite{Dai:PolsAnnilTwoLNDS}.
Local slice construction was originally defined in \cite{Freud:LocalSlice}
in the case $B=\kk3$, and was later generalized in \cite{Dai:LSC}.
\end{remark*}

\begin{proof}[Proof of Corollary \ref{ofoofofoofofoo}]
Let $S = R \setminus \{0 \}$, $\Aeul_i = S^{-1}A_i$ ($i=1,2$) and
$\Beul = S^{-1}B = K \otimes_R B$.
If $D_i \in \lnd(B)$ has kernel $A_i$,
then $S^{-1}D_i \in \lnd( \Beul )$ has kernel $\Aeul_i$;
thus $\Aeul_1, \Aeul_2 \in \klnd( \Beul )$.
Using that $A_1, A_2$ are factorially closed in $B$, we obtain
$\Aeul_1 \cap \Aeul_2 \subseteq K$, so $\ML( \Beul ) \subseteq K$.
The reverse inclusion is trivial
($K^* \subseteq \Beul^* \subseteq \ML( \Beul )$), so $\ML( \Beul ) = K$.
Then $\Beul \in \Dan( K )$ by Theorem~\ref{dwddwdwdwddwd}, so
assertion~(a) is proved.

In \cite[3.3]{Dai:LSC}, one defines a graph $\Klnd(B)$ whose vertex-set
is $\klnd(B)$; then, given $A,A' \in \klnd(B)$, one says that
$A'$ can be obtained from $A$ ``by a local slice construction'' if
there exists an edge in $\Klnd(B)$ joining vertices $A$ and $A'$.
So assertion~(b) of the Corollary is equivalent to the existence of
a path in $\Klnd(B)$ going from $A_1$ to $A_2$.
Paragraph \cite[3.2.2]{Dai:LSC} also defines a subgraph
$\Klnd_R(B)$ of the graph $\Klnd(B)$, and clearly
$A_1, A_2$ are two vertices of $\Klnd_R(B)$;
so, to prove (b), it suffices to show that 
$\Klnd_R(B)$ is a connected graph.
We have $R \in \Reul^{\mbox{\scriptsize in}}(B)$ (cf.\ \cite[5.2]{Dai:LSC})
and consequently (cf.\ \cite[5.3]{Dai:LSC}, using that $B$ is a UFD)
we have an isomorphism of graphs $\Klnd_R(B) \isom \Klnd_K(\Beul)$.
As $\Beul \in \Dan( K )$ by part~(a),
we may apply \cite[4.8]{Dai:LSC} and conclude that
$\Klnd_K(\Beul)$ is connected.  Assertion~(b) is proved.
\end{proof}

The following is a trivial consequence of Corollary~\ref{ofoofofoofofoo}.

\begin{corollary}\label{dwdwdwdwrwdwrdwrwdrwdwr}
Let $B \in \Neul( \bk )$, where $\bk$ is a field of characteristic zero.
Suppose that $B$ has transcendence degree two over $\ML(B)$.
\begin{enumerate}

\item Let $R = \ML(B)$ and $K = \Frac R$.  Then $K \otimes_R B \in \Dan(K)$.

\item If $B$ is a UFD then, for any $A_1, A_2 \in \klnd(B)$,
there exists a finite sequence of local slice constructions which 
transforms $A_1$ into $A_2$.

\end{enumerate}
\end{corollary}


\bibliographystyle{amsplain}

\providecommand{\bysame}{\leavevmode\hbox to3em{\hrulefill}\thinspace}
\providecommand{\MR}{\relax\ifhmode\unskip\space\fi MR }
\providecommand{\MRhref}[2]{%
  \href{http://www.ams.org/mathscinet-getitem?mr=#1}{#2}
}
\providecommand{\href}[2]{#2}

\end{document}